\documentclass[11pt]{amsart}

\usepackage{scrextend}

\makeatletter
\newcommand\footnoteref[1]{\protected@xdef\@thefnmark{\ref{#1}}\@footnotemark}
\makeatother
\newcommand{\lmfdbec}[3]{\href{http://www.lmfdb.org/EllipticCurve/Q/#1/#2/#3}{#1#2#3}}

\usepackage{amsmath,amsfonts,amssymb,amsthm,url,verbatim}
\usepackage[dvips]{graphicx}
\usepackage[croatian,english]{babel}
\usepackage{amsfonts}
\usepackage[utf8]{inputenc}
\usepackage[T1]{fontenc}
\usepackage{amsmath}
\usepackage{amssymb}
\usepackage{latexsym}
\usepackage[usenames,dvipsnames]{xcolor}
\usepackage{soul}

\usepackage{booktabs}
\usepackage{longtable}
\usepackage{multirow}

\usepackage[backref]{hyperref}
\usepackage{cleveref}

\usepackage{ifthen}
\usepackage{graphpap}
\newcommand{\cC}{{\mathcal{C}}}

\newcommand{\Q}{\mathbb{Q}}

\newcommand{\Z}{\mathbb{Z}}

\newcommand{\F}{\mathbb{F}}
\newcommand{\OO}{\mathcal{O}}
\newcommand{\Qbar}{{\overline{\mathbb Q}}}

\newtheorem{tm}{Theorem}[section]
\newtheorem{proposition}[tm]{Proposition}
\newtheorem{lemma}[tm]{Lemma}

\theoremstyle{definition}
\newtheorem{definition}{Definition}
\theoremstyle{remark}

\newtheorem{remark}[tm]{Remark}

\DeclareMathOperator{\Gal}{Gal}

\DeclareMathOperator{\End}{End}

\DeclareMathOperator{\Hom}{Hom}

\title{Torsion of $\Q$-curves over quadratic fields}
\author{Samuel Le Fourn}
\address{Mathematics Institute, University of Warwick, Coventry, CV4 7AL, United Kingdom.}
\email{Samuel.Le-Fourn@warwick.ac.uk}
\author{Filip Najman}
\address{University of Zagreb, Faculty of Science, Department of Mathematics,  Bijeni\v{c}ka cesta 30, 10000 Zagreb, Croatia}
\email{fnajman@math.hr}

\thanks{Le Fourn was supported  by the European Union’s Horizon 2020 research and programme under the Marie Sklodowska-Curie grant agreement No 793646, titled LowDegModCurve. Najman was supported by the QuantiXLie Centre of Excellence, a project
co-financed by the Croatian Government and European Union through the
European Regional Development Fund - the Competitiveness and Cohesion
Operational Programme (Grant KK.01.1.1.01.0004) and by the Croatian Science Foundation under the project no. IP-2018-01-1313.}
\date{\today}

\begin{document}
\begin{abstract}
We determine all the possible torsion groups of $\Q$-curves over quadratic fields and determine which groups appear finitely and which appear infinitely often.
\end{abstract}
\maketitle
\section{Introduction}

In the study of elliptic curves over number fields, describing the possible torsion groups plays an important role. This subject has a long history in which probably the most famous results are Mazur's torsion theorem \cite{mazur2} and Merel's proof of the uniform boundedness conjecture \cite{merel}.

Historically, people have mostly studied the possible torsion groups of all elliptic curves over degree $d$ number fields. One might however be interested in a more precise classification of possible properties of elliptic curves over degree $d$ number fields, in the sense that one wants to find the properties of some subset of all elliptic curves defined over degree $d$ number fields. This is both interesting in its own right and might be useful for applications, especially in Diophantine equations, as the curves that arise there often have some special property. To give two (out of many) examples, Pila \cite{pila} uses results about possible isogenies of non-CM elliptic curves with $\Q$-rational $j$-invariants \cite{naj2} (no such results exist for all elliptic curves over degree $d$ fields) to prove results about some diophantine problems arising out of "unlikely intersections" and Dieulefait and Urroz \cite{du} use properties of $\Q$-curves over quadratic fields to solve the equation $x^4 +dy^2 = z^p$ for $d=2$ and $3$ and for large $p$.


The classification of all possible torsion groups of elliptic curves over quadratic fields was done in two steps, by Kenku and Momose \cite{km} and by Kamienny \cite{kam}. All the possible torsion groups have been determined for the following subsets of all elliptic curves over quadratic fields: for elliptic curves with integral $j$-invariants \cite{msz}, elliptic curves that are base changes of elliptic curves defined over $\Q$ \cite{najman} and  elliptic curves with complex multiplication \cite{clark}.

In this paper we determine the possible torsion groups of $\Q$-curves over quadratic fields. Recall that a $\Q$-curve is an elliptic curve isogenous (over $\overline \Q$) to all of its $\Gal(\overline \Q /\Q)$-conjugates. These properties are obviously satisfied by elliptic curves defined over $\Q$, so $\Q$-curves can be thought of as generalizations of elliptic curves defined over $\Q$. Moreover, Ribet \cite{ribet} proved (assuming Serre's conjecture, which has since been proved \cite{khw, khw2}) that $\Q$-curves are exactly the elliptic curves over number fields that are modular, in the sense of being a quotient of $J_1(N)$ for some $N$.

Torsion groups of $\Q$-curves over quadratic (and higher degree) fields have already been studied by Sairaiji and Yamauchi \cite{ST}, but they imposed a number of additional conditions on their curves: the isogeny $\phi:E\rightarrow E^\sigma$ must be of squarefree degree and defined over the quadratic field and $j(E)\not\in \Q$ (here, as in the remainder of the paper, $\sigma$ denotes the generator of $\Gal(K/\Q)$). Moreover, they determine only the possible orders $N$ of points in $E(K)_{tors}$, for a $\Q$-curve over a quadratic field $K$ with the additional property that $N$ is divisible by the degree of the isogeny $\phi:E\rightarrow E^\sigma$.

Throughout the paper $\cC_n$ will denote a cyclic group of order $n$. Our main result is the following theorem.
\begin{tm}
\label{maintm}
Let $E$ be a $\Q$-curve defined over a quadratic field $K$. Then $E(K)_{tors}$ is isomorphic to one of the following groups
\begin{align*}
&\cC_n, \text{ where } 1\leq n \leq 18,\ n\neq 11,17\\
&\cC_2 \times \cC_{2n}, \text{ where } n=1, \ldots, 6,\\
&\cC_3 \times \cC_{3n}, \text{ where } n=1,2,\\
&\cC_4 \times \cC_{4}.
\end{align*}
There are infinitely many $\Q$-curves with each of these torsion groups, except for $\cC_{14}$ and $\cC_{15}$ of which there are finitely many.
\end{tm}
\begin{remark}
For a precise definition of what we mean by "finitely many" and "infinitely many", see \Cref{def:fin}.
\end{remark}

\Cref{prop:11}, \Cref{prop:15} and \Cref{prop:14}, together with the results presented in \Cref{sec:2} will prove \Cref{maintm}.

The only group that can be a torsion group of any elliptic curve over a quadratic field, but \textit{not} of a $\Q$-curve is $\cC_{11}$.
Note that some authors define $\Q$-curve to be curves without complex multiplication (CM). We will allow $\Q$-curves to have CM, although none of the results of the paper would be changed if we restricted to non-CM $\Q$-curves.

A surprising amount of different ingredients (which will be explained in detail in the next section), of which most are very recent, go into the proof of \Cref{maintm}: the results of Sairaiji and Yamauchi about central $\Q$-curves \cite{ST}; classification of possible torsion groups of elliptic curves over quadratic fields that are base changes by Najman \cite{najman}; and of those with integral $j$-invariants by M\"uller, Str\"oher and Zimmer \cite{msz}; Le Fourn's results about the surjectivity of mod $p$ Galois representations attached to quadratic $\Q$-curves and consequences on their reduction types \cite{lf}; the results of Bosman, Bruin, Dujella and Najman showing that all elliptic curves with $\cC_{13}$, $\cC_{16}$ and $\cC_{18}$ torsion are $\Q$-curves \cite{bbdn} and Bars' results about which classical modular curves $X_0(n)$ have infinitely many quadratic points \cite{bars}.

\section{Definitions and useful results form other papers}
\label{sec:2}

Although it is probably clear to the reader what we mean by saying that there are finitely many $\Q$-curves with some torsion group "up to isomorphism", we give a rigorous definition below.

\begin{definition}
\label{def:fin}
We say that there are finitely many elliptic curves over number fields of degree $d$ with a property $P$, if there are finitely many pairs $(E,K)$ with that property $P$ up to isomorphism, where isomorphism means the following: a pair $(E_1,K_1)$ with $E_1/K_1$ an elliptic curve is isomorphic to $(E_2,K_2)$, where $E_2/K_2$ is an elliptic curve, if there exists an isomorphism of fields $\phi_1:K_1\rightarrow K_2$ and a $K_2$-isomorphism of elliptic curves $\phi_2:\phi_1(E_1)\rightarrow E_2$. Here $\phi(E_1)/K_2$ is the elliptic curve obtained by taking a model of $E_1$ over $K_1$ and then mapping the coefficients of $E_1$ to $K_2$ via $\phi_1$, thereby obtaining a model of $\phi(E_1)$ over $K_2$.
\end{definition}

\begin{definition}
Let $E$ be a $\Q$-curve, $\phi_{\sigma}:E\rightarrow E^\sigma$ be an isogeny of degree $d_\sigma$. We then say that $E$ is a $\Q$-curve of degree $d_\sigma$. If $d_\sigma$ is square-free, we say that $E$ is a central $\Q$-curve.
\end{definition}

Throughout the paper $d_\sigma$ will always denote the degree of a $\Q$-curve, often without explicit mention.

The modular curve $X_0^+(N)=X_0(N)/w_N$ is a moduli space whose non-cuspidal points correspond to a set of two elliptic curves which are $N$-isogenous to each other. Points in $X_0^+(N)(\Q)$ correspond to a pair of $\Q$-curves of degree $N$ defined over a quadratic field. Quadratic central $\Q$-curves correspond to rational points on $X_0^+(N)$, where $N$ is square free.

First, it is useful to know which groups can appear as torsion groups of all elliptic curves over quadratic fields.

\begin{tm}\cite{kam,km}
\label{svi}
Let $E$ be an elliptic curve defined over a quadratic field $K$. Then $E(K)_{tors}$ is isomorphic to one of the following groups
\begin{align*}
&\cC_n, \text{ where } n=1, \ldots,\ldots 16, 18,\\
&\cC_2 \times \cC_{2n}, \text{ where } n=1, \ldots, 6,\\
&\cC_3 \times \cC_{3n}, \text{ where } n=1,2,\\
&\cC_4 \times \cC_{4}.
\end{align*}
\end{tm}
\begin{remark}
For all groups $T$ in the list there exist infinitely many (up to isomorphism) pairs $(E,K)$ such that $E(K)_{tors}\simeq T$.
\end{remark}

The following theorem tells us which torsion groups are possible for elliptic curves with integral $j$-invariant over quadratic fields.
\begin{tm}\cite[Theorem 4]{msz}
\label{tm:int}
Let $E$ be an elliptic curve with integral $j$-invariant over a
quadratic field $K$. Then, up to isomorphism, the torsion group of $E$ over $K$ is one of the following groups
\begin{align*}
&\cC_n, \text{ where } n=1, \ldots,\ldots 8,10,\\
&\cC_2 \times \cC_{2n}, \text{ where } n=1, 2, 3,\\
&\cC_3 \times \cC_{3}.\\
\end{align*}
\end{tm}

To prove \Cref{maintm}, we need to show:
\begin{itemize}
  \item[1)] That $\cC_{11}$ does not appear as the torsion of a $\Q$-curve over a quadratic field.
  \item[2)] That each group from \Cref{svi} except $\cC_{11}$ does appear as a torsion group of a $\Q$-curve, and determine whether it appears infinitely often.
\end{itemize}

The first place where one wants to start looking for $\Q$-curves is among elliptic curves defined over $\Q$.

\begin{tm}\cite[Theorem 2]{najman}
\label{tm:bc}
Let $E$ be an elliptic curve defined over $\Q$ and let $K$ be a quadratic field. Then $E(K)_{tors}$ is isomorphic to one of the following groups
\begin{align*}
&\cC_n, \text{ where } n=1, \ldots,10,12,15,16,\\
&\cC_2 \times \cC_{2n}, \text{ where } n=1, \ldots, 6\\
&\cC_3 \times \cC_{3n}, \text{ where } n=1,2,\\
&\cC_4 \times \cC_{4}.
\end{align*}
For all groups $T$ in the list, except for $\cC_{15}$, there exist infinitely many (up to isomorphism) pairs $(E,K)$, with $E/\Q$, such that $E(K)_{tors}\simeq T$. The elliptic curves with Cremona label 50b1 and 50a3 have $15$-torsion over $\Q(\sqrt 5)$, and 50b2 and 450b4 have $15$-torsion over $\Q(\sqrt{-15})$; these are the only elliptic curves defined over $\Q$ having non-trivial $15$-torsion over any quadratic field.
\end{tm}

Bosman, Bruin, Dujella and Najman \cite[Theorem 4.1]{bbdn} proved that all elliptic curves over quadratic fields with torsion $\cC_{13}$, $\cC_{16}$ and $\cC_{18}$ are $\Q$-curves. Moreover, in \cite[Theorem 1.1.]{BN2} it is shown that all elliptic curves with torsion $\cC_{16}$ over quadratic fields are base changes of elliptic curves over $\Q$. Recall that there exist infinitely many elliptic curves with these torsion groups.

Combining these results, we see that all that remains is to show that $\cC_{11}$ does not appear, that $\cC_{14}$ does appear, but only finitely often, and that $\cC_{15}$ appears finitely often.

The relationship between $\Q$-curves and central $\Q$-curves is nicely explained in the following proposition.
\begin{proposition}\cite[Proposition 1.3.]{lf}
\label{prop:central}
For every $\Q$-curve $E$ without CM defined over a number field $K$, there exists an isogeny $\psi:E\rightarrow E'$ towards a central $\Q$-curve $E'$ defined over $K$ such that the degree of $\psi$ divides $d_\sigma$ and is squarefree.
\end{proposition}

Sairaji and Yamauchi proved the following result (which we specialize to quadratic fields) that we will find useful.

\begin{tm}
\label{thm:st}
\cite[Proposition 3.4.]{ST}
Let $E$ be a central $\Q$-curve of degree $d_\sigma$ defined over a quadratic field $K$ such that $j(E)\not \in \Q$ with a point of prime order $N$ in $E(K)$. If $N$ divides $d_\sigma$, then $N$ is either $2$ or $3$.
\end{tm}


We will also need the following result of Le Fourn.

\begin{proposition}\cite[Proposition 2.3]{lf}
\label{tm:lf}
Let $K$ be a quadratic field and $E/K$ a $\Q$-curve of degree $d_\sigma$. Let $p\geq 11$, $p\neq 13,17,41$, where $p$ is a prime not dividing $d_\sigma$, such that the mod $p$ Galois representation attached to $E$ is reducible. Then $E$ has potentially good reduction at all prime ideals of $\OO_K$.
\end{proposition}

We will make use of which modular curves $X_0(n)$ have infinitely many quadratic points when proving the finiteness of quadratic $\Q$-curves with $\cC_{14}$ and $\cC_{15}$ torsion.

\begin{tm}\cite[Theorem 4.3]{bars}
\label{tm:bars}
If $X_0(n)$ is of genus $\geq 2$ and has infinitely many quadratic points then $n$ is one of the following values:
$$22,23,26,28,29,30,31,33,35,37,39,40,41,43,46,47,48,50,53,59,61,65,71,79,83,89,101\text { or } 131.$$

\end{tm}

\section{There are no quadratic $\Q$-curves with $\cC_{11}$ torsion.}

Before proceeding with the proof of our results, we prove a general lemma that we will need.
\begin{lemma}\label{lem:izogenija}
Let $F$ be a number field, let $E_1,E_2$ be elliptic curves both defined over $F$ without complex multiplication, and suppose $E_1$ and $E_2$ are isogenous over $\overline F$. Then there exists a quadratic extension $F'/F$ such that $E_1$ and $E_2$ are isogenous over $F'$.
\end{lemma}
\begin{proof}
Let $f_1\in\Hom(E_1,E_2)$ be an isogeny of degree $n$. We can suppose that $f_1$ is cyclic, as otherwise it would be a composition of multiplication-by-$m$  for some $m$ on $E_1$ and some cyclic isogeny. The absolute Galois group group $G_F:=\Gal(\overline F/F)$ acts on $\Hom(E_1,E_2)$ and preserves the degree, i.e $\deg f_1^\sigma =n$ for all $\sigma \in G_F$.

Let $f_2:=f_1^\sigma$. As $\widehat{f_2}\circ f_1\in \End E_1=\Z$, it follows that $\widehat f_2\circ f_1 =[m]_{E_1}$ for some integer $m$. By taking degrees, we see that $m=n$ or $m=-n$. Since the dual isogeny is unique, we have $f_2=f_1$ or $f_2=[-1]\circ f_1$. Hence, we have a homomorphism from $G_F$ into $\{\pm 1\}$. The kernel of this map is an index 2 subgroup of $G_F$, hence is $G_{F'}$ for some quadratic extension $F'/F$. It follows that $f_1$ is defined over $F'$.
\end{proof}

We first deal with the torsion group $\cC_{11}$.
\begin{proposition}
\label{prop:11}
There does not exist a $\Q$-curve with torsion $C_{11}$ over a quadratic field.
\end{proposition}
\begin{proof}
Suppose such a curve, $E/K$ exists, with $E(K)_{tors}\simeq C_{11}$, where $E$ is a $\Q$-curve.  The proof will proceed by proving a series of claims, of which each is used in the next. The strategy is to show that we can assume that $E$ satisfies the assumptions of \Cref{thm:st}, which we then use to show that the curve satisfies the assumptions of \Cref{tm:lf}, which we then use to prove the result.

{\sc Claim 1:} $E$ does not have complex multiplication.\\
This follows from \cite[Section 4]{clark}, as an elliptic curve with CM cannot have a point of order 11 over a quadratic field.\\

{\sc Claim 2:} $j(E)\not \in \Q$.\\
Suppose $j(E) \in \Q$. First note that if $j(E)\in \Q$, then it does not necessarily mean that $E$ is a base change of an elliptic curve defined over $\Q$, but we do know that there exists an elliptic curve $E'/\Q$ such that $j(E)=j(E')$. Since, by Claim 1, $j(E)\neq0$ or $1728$, it follows that $E'$ is a quadratic twist (over $K$) of $E$. 
So $E'=E^d$, for some square-free integer $d$. It follows that $E^d$ has a point of order 11 over a quartic field (since $E$ and $E^d$ become isomorphic over $K(\sqrt d)$), contradicting \cite[Corollary 8.7.]{gn} or \cite[Corollary 1.1.]{loz}, proving the claim.\\

{\sc Claim 3:} $d_\sigma$ is coprime to 11.\\
If $E$ was a central $\Q$-curve, then by Claim 2, it would satisfy the assumptions of \Cref{thm:st}, and hence prove the claim. Suppose $E$ is not a central $\Q$-curve. Then, it is isogenous to a central $\Q$-curve $E'$ by a degree $d|d_\sigma$ isogeny by Proposition \ref{prop:central}. Note that the isogeny in question does not have to be defined over $K$, but this causes no problem because if $E$ and $E'$ are two curves defined over $K$ which do not have complex multiplication and are isogenous over $\overline K$, then $E$ is isogenous over $K$ to a quadratic twist $(E')^d$ of $E'$ by Lemma \ref{lem:izogenija}, and this quadratic twist is also a central $\Q$-curve. Thus, we can assume that $E$ is $K$-isogenous to a central $\Q$-curve $E'$ by an isogeny of degree $d|d_\sigma$.

We have the following isogeny diagram:
$$E\xrightarrow{\psi_1} E' \xrightarrow{\psi_2} (E')^\sigma\xrightarrow{\psi_3} E^\sigma.$$
We may assume all of these isogenies are defined over $K$ after replacing by a twist if necessary.
So, $\psi_1$ and $\psi_3$ which are both of the same degree  $d$, and let $\psi_2$ be of degree $d'$. Then we have $d_\sigma = d^2d'$.

If $11$ does not divide $d$, $E'(K)$ has a point of order 11, and from \Cref{thm:st}, we see that 11 also does not divide $d'$, so $11$ does not divide $d_\sigma$.

Suppose $11$ does divide $d$. But then $E$ corresponds to a non-CM point in $X_0^+(N)(\Q)$, where $N$ is divisible by 121, contradicting \cite[Theorem (0.1) (i)]{momose}, proving the claim.

\subsection{Proof of the proposition}
By what we have proved, $d_\sigma$ is coprime to 11. By \Cref{tm:lf}, $E$ has potentially good reduction at all primes of $\OO_K$ since $d_\sigma$ is coprime to 11 and the mod 11 representation is reducible. Let $\wp$ be the prime of $K$ over 2.

Suppose $E$ has good reduction at $\wp$. The residue field $k(\wp)$ of $\wp$ is either $\F_2$ or $\F_{4}$, and since $E(K)_{tors}$ injects into $E(k(\wp))$ (since $E(K)$ has no 2-torsion), we arrive to a contradiction with the Hasse bound.

We conclude that $E$ has additive reduction at $\wp$. Denote by $\widetilde E$ the reduction of $E$ mod $\wp$, i.e. the special fibre of the N\'eron model of $E$ at $\wp$. Reduction mod $\wp:$ $E(K)\rightarrow \widetilde E(k(\wp))$ is injective on the odd order torsion of $E(K)$, hence the image of the point $P\in E(K)$ of order 11 has order 11 in $\widetilde E(k(\wp))$. But by the Kodaira-N\'eron theorem \cite[Theorem 6.1, p.200]{silverman}, $\widetilde E(k(\wp))$ is a product of the additive group of $k(\wp)$ and a group of order $\leq 4$, which is a contradiction.

\end{proof}

\section{The finiteness of $\Q$-curves with $\cC_{14}$ and $\cC_{15}$}

We first show that there exist only finitely many elliptic curves with $\cC_{14}$ and $\cC_{15}$ over quadratic fields with $j(E)\in \Q$.

\begin{lemma}
There exist no elliptic curves with complex multiplication with torsion $\cC_{14}$ and $\cC_{15}$ over quadratic fields.
\end{lemma}
\begin{proof}
This follows immediately from \cite{clark}.
\end{proof}

\begin{lemma}\cite[Lemma 5.5]{mr}
\label{lem:finite}
For each $j\in K$, there exist only finitely many elliptic curves $E/K$ (up to $K$-isomorphism) with an odd order point and $j(E)=j$.
\end{lemma}

\begin{lemma}
\label{lem:finite2}
There exist finitely many elliptic curves over quadratic fields $E/K$ such that $j(E)\in \Q$ and such that $E(K)_{tors}\simeq \cC_{15}$ and no such curves with $E(K)_{tors}\simeq \cC_{14}$.
\end{lemma}
\begin{proof}
Suppose that there exists an elliptic curve $E/K$ over a quadratic field with $E(K)_{tors}\simeq \cC_{14}$ and $j(E)\in \Q$. As in the proof of Claim 2 in the proof of Proposition \ref{prop:11}, we can conclude that there exists a quadratic twist $E^d$ of $E$ such that $E^d$ is defined over $\Q$ and $E^d(K(\sqrt{d}))$ contains $\cC_{14}$. As $E^d$ and $E$ become isomorphic over a quadratic extension $K(\sqrt d)$ of $K$, it follows that $E^d(K(\sqrt{d}))$ contains $\cC_{14}$ as a subgroup. But this is impossible by \cite[Corollary 8.7]{gn}.

Let now $E/K$ be an elliptic curve over a quadratic field with $E(K)_{tors}\simeq \cC_{15}$ and $j(E)\in \Q$. Let $E^d$ be a quadratic twist of $E$ such that $E^d$ is defined over $\Q$ and $E^d(K(\sqrt{d}))$ contains $\cC_{15}$. Let $G(p)$ denote the image of the mod $p$ representation attached to $E^d$.

From \cite[Table 1]{gn}, we see that if $E^d$ gains a point of order $3$ over a number field whose degree divides $4$, then $G(3)$ is contained in either the Borel subgroup or the normalizer of the split Cartan subgroup. If $E^d$ gains a point of order $5$ over a number field whose degree divides $4$, then $G(5)$ is contained in the Borel subgroup. So if $G(3)$ is contained in a Borel subgroup, then $E^d$ has a $15$-isogeny over $\Q$; there are only finitely many such $j$-invariants (see \cite{mazur2}), and hence only finitely many such curves by \Cref{lem:finite}. If $G(3)$ is contained in the normalizer of the split Cartan subgroup, then $E^d$ corresponds to a rational point on the modular curve $X(s3,b5)$ (using the notation of \cite{siksek}) which is an elliptic curve of conductor 15 with finitely many points \cite[Lemma 5.7]{siksek}.

It remains to prove that for each fixed $j$, there exist only finitely many quadratic fields $K$ such that there exists an elliptic curve $E/K$ with $j(E)=j$ and $E(K)_{tors}\simeq \cC_{15}$. Suppose $E/K$ is such a curve. Let $E_0/\Q$ be an elliptic curve with $j(E_0)=j$, and $E_0=E^\delta$ for some $\delta \in K^\times$. If $E_0$ gains a 15-torsion point over a quadratic field, then by \Cref{tm:bc}, $E_0$ is one of 4 curves with Cremona label $50b1$, $50a3$, $50b2$ or $450b4$ and $K$ is either $\Q(\sqrt{5})$ or $\Q(\sqrt{-15})$. Suppose now $E_0$ is not one of those curves. Then $E_0$ gains a point of order 15 over $K(\sqrt \delta)$, but not over a subfield of $K(\sqrt \delta)$. This means that $K(\sqrt \delta)\subseteq \Q(E_0[15])$. But $\Q(E_0[15])$ is a number field, so has only finitely many quartic subfields. We conclude that there are finitely many  quadratic fields $K$ with our property (that there exists a curve $E'$ with $j(E')=j$ and torsion $\cC_{15}$ over $K$). Now for each such $K$, there exist only finitely many (up to $K$-isomorphism) elliptic curves with $j(E')=j$ with torsion $\cC_{15}$ over $K$ by \Cref{lem:finite}, proving the lemma.
\end{proof}

From now until the end of this section, to avoid repetition, we suppose that all elliptic curves have $j(E)\notin \Q$. From this it follows that $d_\sigma\neq 1$, as $d_\sigma=1$ would imply that $j(E)\in \Q$.

\begin{proposition}
\label{prop:15}
There exists finitely many $\Q$-curves $E$ over quadratic fields $K$ such that $E(K)_{tors} \simeq \cC_{15}$.
\end{proposition}
\begin{proof}
By \Cref{lem:finite} and \Cref{lem:finite2},  it is enough to show that there are finitely many quadratic $j$-invariants $j$ such that there exists an elliptic curve $E$ over a quadratic field $K$ with $j(E)=j\in K\backslash \Q$ with torsion $E(K)_{tors}\simeq\cC_{15}$ over $K$.

First suppose that $E$ is a central $\Q$-curve. Then by \cite[Theorem 1.2]{ST}, it is not possible that $d_\sigma|15$. On the other hand, by Lemma \ref{appendix:lem2}, there are finitely many central $\Q$-curves such that $d_\sigma$ does not divide 15.




Suppose now that $E$ is not a central $\Q$-curve. By \Cref{prop:central}, $E$ is isogenous to a central $\Q$-curve $E'$.

If $E'(K)$ has a point of order 15 over $K$, by what we have shown, $E'$ is one of finitely many such central $\Q$-curves. In each isogeny class over $K$ there are finitely many elliptic curves, so we can conclude that there are finitely many curves $E$ satisfying these assumptions.

If $E'(K)$ does not have a point of order $15$, then there must exist an isogeny from $E$ towards $E'$ of degree $3$ or $5$ (or both). By using the same argumentation as in the proof of \Cref{prop:11}, we see that if the isogeny $\psi_1:E\rightarrow E'$ is of degree $d$, then $E$ has an isogeny of degree $d^2$. If $3$ divides $d$, then $E$ has both an $9$-isogeny and a $5$-isogeny, hence a 45-isogeny. On the other hand if $5$ divides $d$, using the same argumentation, we see that $E$ has a $75$-isogeny. In both cases, there are only finitely many quadratic points on $X_0(45)$ and $X_0(75)$ by \Cref{tm:bars}, so again there can be only finitely many curves $E$ satisfying these conditions.

\end{proof}

\begin{proposition}
\label{prop:14}
There exist finitely many $\Q$-curves $E$ over quadratic fields $K$ such that $E(K)_{tors} \simeq \cC_{14}$.
\end{proposition}
\begin{proof}
By \Cref{lem:finite} and \Cref{lem:finite2},  it is enough to show that there are finitely many quadratic $j$-invariants $j$ such that there exists an elliptic curve $E$ over a quadratic field $K$ with $j(E)=j\in K\backslash \Q$ with torsion $E(K)_{tors}\simeq \cC_{14}$ over $K$.

Let $E$ be a central $\Q$-curve such that $d_\sigma|14$. First note that By \Cref{thm:st}, we have $d_\sigma=2$. This means that $E$ corresponds to a $\Q$-rational point on the curve $X=X_0(14)/w_2$. By \cite[Appendix, p.29]{gl} we see that $X$ is an elliptic curve. As $X$ is $2$-isogenous to $X_0(14)$, which has rank 0 over $\Q$, it follows that $X(\Q)$ is finite.

Suppose now that $E$ is a $\Q$-curve such that $d_\sigma$ does not divide $14$. 
By Lemma \ref{appendix:lem3} it follows that, with finitely many exceptions, $E$ has potentially good reduction everywhere, and in particular the $j$-invariant is integral, giving a contradiction with \Cref{tm:int}.



Suppose now that $E$ is not a central $\Q$-curve. By \Cref{prop:central}, $E$ is isogenous to a central $\Q$-curve $E'$. We can suppose that $E'$ does not have a 14-torsion point over $K$, for otherwise, by what we have already proved, $E$ lies in one of finitely many isogeny classes containing a central $\Q$-curve with a point of order 14. We can see that there is an isogeny $E\rightarrow E'$ whose degree is divisible by $7$; using the same argumentation as in the proof of \Cref{prop:15}, we see that $E$ has a $49$-isogeny. There are only finitely many quadratic points on $X_0(49)$ by \Cref{tm:bars}, so we have that $E$ is one of finitely many such curves.


\end{proof}

\begin{remark}
In \cite[p.467]{ST}, the authors say in passing that the curve
$$y^2+(2+\sqrt{-7})xy+(5+\sqrt{-7})y=x^3+(5+\sqrt{-7})x^2$$
is the unique central $\Q$-curve over a quadratic field of degree dividing 14 with 14-torsion, but they give no proof of the claim.
\end{remark}

\noindent {\sc Acknowledgements.} We are grateful to the anonymous referee for his/her useful suggestions and careful review.

\newpage

\appendix

\section{Integrality of $j$-invariants of $\Q$-curves with $\cC_{14}$ and $\cC_{15}$ torsion}
\begin{center}

\end{center}
\bigskip
The proofs of Propositions \ref{prop:15} and \ref{prop:14} require using Mazur's method to imply potentially good reduction for elliptic curves at (almost) all prime ideals. To be more precise, we will prove the following lemmas, which are needed in the proofs of Propositions \ref{prop:15} and \ref{prop:14}.

\begin{lemma}
\label{appendix:lem1}
		Let $n=14$ or $n=15$, $K$ be a quadratic field and $E$ a central $\Q$-curve without CM defined over $K$ and of degree $d$ not dividing $n$. We denote by $E^\sigma$ the conjugate of $E$ by the nontrivial automorphism $\sigma$ of $K$, $\mu$ an isogeny from $E$ to $E^\sigma$ of degree $d$ and $C_d$ its kernel. Define $n'=n/(n,d)$ (always prime to $d$) and assume that $E$ has a subgroup $C_{n'}$ of order $n'$ defined over $K$, and that $C_{n'}^\sigma = \mu(C_{n'})$ as subgroups of $E^\sigma$. Then:
		
		\begin{itemize}
			\item If $n=14$ then $E$ has potentially good reduction at every prime ideal of $K$.
			
			\item If $n=15$ then $E$ has potentially good reduction at every prime ideal not dividing 2.
			
		\end{itemize}
	
\end{lemma}

\begin{proof}

	Let $E$ be an elliptic curve satisfying the hypotheses of the lemma. Let us first assume that $d$ is prime to $n$, hence $n'=n$. The Atkin-Lehner involution $w_d$ acts on the noncuspidal points of $X_0(dn)$ by
	\[
	w_d(E,C_d,C_n) = (E/C_d,E[d]/C_d,(C_n+C_d)/C_d),
	\]
	and by hypothesis one has $P^\sigma = w_d \cdot P$ for the point $P=(E,C_d,C_n)$ coming from our choice of elliptic curve. Let us then define $\pi : X_0(dn) \rightarrow X_0(n)$ to be the natural degeneracy morphism, and $f = \pi + \pi \circ w_d$, so that $f(P)$ is clearly a rational point on $X_0(n)$.
	
	Now, by classical arguments (see the complex case, or more precisely look at the $q$-expansions at the cusps), the degeneracy morphism $X_0(d) \rightarrow X(1)$ is a formal immersion for all prime ideals at the cusp $\infty \in X_0(d)$ and ramified at every other cusp, and this holds similarly for $\pi$: amongst the cusps of $X_0(dn)$ above $\infty \in X_0(n)$, the morphism $\pi$ is a formal immersion at $\infty \in X_0(dn)$ and ramified at the other cusps. Consequently, $f$ is a formal immersion at the cusp $\infty \in X_0(dn)$ for all prime ideals.
	
	Now, let us assume that the point $P$ associated to our $\Q$-curve reduces to a cusp at some prime $\lambda$ dividing the prime number $\ell>2$. As the Atkin-Lehner involutions act transitively on the cusps because $dn$ is squarefree, we can assume that $P$ reduces to the cusp  $\infty$ modulo $\lambda$ (in particular, its reduction lands in the smooth part of the N\'eron model of $X_0(dn)$), and it does not change the fact that $w_d \cdot P = P^\sigma$ as the Atkin-Lehner involutions acts through an abelian group and are defined over $\Q$. Now, $f$ is a formal immersion at $\infty$, sends $P$ to a rational point, and the group of rational points of $X_0(n)$ is finite. One can thus apply Raynaud's lemma in the form given in \cite[Proposition 2.4]{lf} as $1< \ell -1$, and $P$ must be equal to $\infty$ which is obviously a contradiction. Therefore, $E$ has potentially good reduction at all prime ideals of $\mathcal{O}_K$ not above 2.
	
	Let us now consider the case where the prime ideal $\lambda$ is above $\ell=2$. Following step by step the previous argument is straightforward up to the moment we apply \cite[Proposition 2.4]{lf}. In this case, for $\ell=2$, $1=\ell-1$ therefore one needs to know that $f(P)$ does not generate a subgroup scheme $\mu_2$ over $\Z_2$, i.e. it is not a point or order 2 reducing to 0 modulo 2. The modular curve $X_0(14)(\Q)$ is \lmfdbec{14}{a}{6} in the LMFDB \cite{lmfdb}, $X_0(14)(\Q)\cong \Z/6\Z$, and we readily check that the unique point of order 2 does not reduce to 0 in the reduction modulo 2 of this elliptic curve (which is nonsplit multiplicative in this case). Therefore, we can again apply \cite[Proposition 2.4]{lf}, which completes the proof for all prime ideals in the case $n=14$.
	
	Assume now that $(d,n) \neq 1$ (recall we have assumed that $d$ does not divide $n$ therefore $n'$ is a prime number). By the hypothesis and with the notations of the Lemma, the triple $(E,C_d,C_{n'})$ defines a quadratic point $P$ on $X_0(dn')$ (which is of squarefree level by construction), which furthermore satisfies
	\[
	P^\sigma = w_d \cdot P.
	\]
	We define the morphism $\pi : X_0(dn') \rightarrow X_0(n)$ as the natural degeneracy morphism which functorially sends a triple $(E,C_d,C_{n'})$ to the unique cyclic subgroup of order $n$ of $C_d+C_{n'}$, and again we define $f=\pi + \pi \circ w_d$. The degeneracy morphism $X_0(d) \rightarrow X(1)$ is again ramified at all cusps except $\infty$, and a formal immersion at $\infty$ for all primes. Let us now assume that $E$ has potentially bad reduction at some prime ideal $\lambda$ of $\mathcal{O}_K$. One can again use the Atkin-Lehner involutions to assume that $P$ reduces to the pole  $\infty \in X_0(dn')$. Now, the morphism $\pi$ is ramified at the image of this pole by $w_d$. Indeed, writing $d'=d/(n,d)$, cusps of $X_0(dn')$ can be seen as triples of cusps of respectively $X_0(d'),X_0((n,d)),X_0(n')$ in a way compatible with the action of Atkin-Lehner involutions $w_{d'}, w_{(n/d)},w_{n'}$ (this can be worked out straightforwardly from the definitions of Atkin-Lehner involutions, e.g. \cite[Lemma 9]{atkin}, the levels here being squarefree). In particular, $w_d \cdot(\infty,\infty,\infty) = (0,0,\infty)$ and $\pi$ sends $(0,0,\infty)$ to $(0,\infty)$, whence it is ramified at this point (as $X_0(d') \rightarrow X(1)$ is ramified at 0). It is, as before, a formal immersion at $\infty$ for the $\lambda$, so we again have that $f$ is a formal immersion at the cusp $\infty \in X_0(dn')$ (it works even though the images of $\pi(\infty)$ and $\pi(w_d(\infty))$ are not the same by a straightforward computation of the differential), and can reproduce the argument above word for word, which concludes the proof of the Lemma.
\end{proof}

To apply Lemma \ref{appendix:lem1} to the proofs of Propositions \ref{prop:15} and \ref{prop:14} we still need to show that we can suppose that $C_{n'}^\sigma = \mu(C_{n'})$ in the sense that there will be finitely many $\Q$-curves that do not satisfy these assumptions.

\begin{lemma}
\label{appendix:lem3}
Let all the assumptions be as in Lemma \ref{appendix:lem1}, except we do not assume that $C_{n'}^\sigma = \mu(C_{n'})$. Then for all but finitely many such $\Q$-curves, the following is true:
\begin{itemize}
			\item If $n=14$ then $E$ has potentially good reduction at every prime ideal of $K$.
			
			\item If $n=15$ then $E$ has potentially good reduction at every prime ideal not dividing $2$.
			
		\end{itemize}
\end{lemma}

\begin{proof}
Assume that $E$ is a central $\Q$-curve of degree $d$, defined over $K$ quadratic and admitting a point of order $n=14$ or 15 over $K$.  We define 
$\mu$ to be the isogeny of degree $d$ from $E$ to its conjugate. By assumption $d$ does not divide $n$, so $(n,d) \in \{1, 2,3,5,7 \}$.

Let us first assume that $(n,d) \neq 1$ so that $n'$ is prime. Let $C_n$ be the cyclic subgroup of order $n$ generated by a $n$-torsion point, and $C_{n'}$ its $n'$-torsion subgroup. As $d$ is prime to $n'$, $\widehat{\mu}(C_{n'}^\sigma)$ is a subgroup of order $n'$ of $E$ defined over $K$, and if it is not $C_{n'}$, the elliptic curve $E$ is thus endowed with two distinct subgroups of order $n'$ and one subgroup of order $n/n'$, all defined over $K$, from which we obtain a point of $X_0(nn')$. Such a modular curve has only finitely many quadratic points unless $n=14$ and $n'=2$ by Theorem \ref{tm:bars}, but this implies that there is a point of order 7 on the $\Q$-curve, which is impossible by \Cref{thm:st} (note that one can suppose $j(E)\not\in \Q$ by Lemma \ref{lem:finite2}).

Consider now the case $(d,n)=1$. The previous argument holds up until considering the possibility that $\widehat{\mu}(C_{n}^\sigma) \neq C_n$. In this case, $E$ is endowed with an additional $k$-subgroup for a prime number $k$ dividing $n$, which defines (after possibly considering an isogenous elliptic curve) a quadratic point of $X_0(kn)$ where $kn \in \{28,49,45,75,196,225\}$. For all cases but the first, there are only finitely many quadratic points on those curves by \Cref{tm:bars}. In the first case, it follows from \cite{BN} that, up to finitely many exceptions, the quadratic points of $X_0(28)$ correspond to $\Q$-curves of degree 7, which cannot happen here as we assumed $(d,n)=1$.
	

	
\end{proof}

	Unfortunately, one cannot apply Lemma \ref{appendix:lem3} to prove Proposition \ref{prop:15} directly, as there exists the issue of the $j$-invariant possibly not being integral at the primes above 2. For $n=15$, the group of rational points of $X_0(n)$ is isomorphic to $\Z/4\Z \times \Z/2\Z$, the reduction modulo 2 is good, and there \textit{is} one point generating a $\mu_2$: in the LMFDB database, a Weierstrass equation for $X_0(15)$ is given by the elliptic curve defined by the affine equation
	\[
	y^2 + xy + y = x^3 +x^2 - 10x-10,
	\]
	and this point is then given by the coordinates $(-13/4,9/8)$. A natural workaround would be to use here the strengthened hypothesis that there is a point of order 15 defined over $K$, not merely a subgroup (hence to argue with $X_1(15)$ instead of $X_0(15)$). The problem is that it does not behave that well with respect to the structure of $\Q$-curves, in particular there does not seem to be a clear general way to define a morphism $X_0(d) \times X_1(n) \rightarrow X_1(n)$ which sends our point associated to $E$ to a rational point and is a formal immersion at the cusp $\infty$ at the same time.
	
	After some time of reflection, we could not come to an argument making use of an additional property of $P$ (relatively to more general quadratic points of $X_1(n)$) to ensure that $f(P)$ is not this point. We instead settle this case by applying  Runge's method.

\begin{lemma} \label{appendix:lem2} There exist finitely many central $\Q$-curves $E$ over quadratic fields, without CM, with a point of order $15$ and of degree $d$ not dividing $15$.\end{lemma}

\begin{proof}
	
	

	

	
Let $E/K$ be such a curve, where $K$ is a quadratic field. Notice that $E$ must be semistable because it has a point of order 15 defined over $K$, hence it has good reduction at every prime ideal not dividing 2 by Lemma \ref{appendix:lem3}. Consider the set of places $S$ of $K$ made up of the infinite places and the places above 2, so that $E$ defines an $\mathcal{O}_{K,S}$-integral point of $X_1(15)$. Clearly, $|S|<5$, which is less than the number of orbits of cusps of $X_1(15)$ by $\operatorname{Gal}(\Qbar/\Q)$ (hence $\operatorname{Gal}(\overline{K}/K)$).

This allows us to apply Runge's method to our integral point, and thus get an \textit{absolute} bound on the height of its $j$-invariant - see \cite[Theorem 1.2]{BP11} for details. In particular, there are only finitely many such elliptic curves, which is what we wanted to prove.
\end{proof}

\end{document}